\documentclass[12pt]{amsart}

\usepackage[margin=1in,left=0.85in,right=0.85in]{geometry}
\usepackage{amsmath,amssymb,amsfonts,mathtools}
\usepackage{hyperref}
\usepackage{tikz}
\usepackage{bbm}
\usepackage{xcolor}

\theoremstyle{definition}
\newtheorem{theorem}{Theorem}[section]
\newtheorem{lemma}[theorem]{Lemma}
\newtheorem{proposition}[theorem]{Proposition}

\newtheorem{definition}[theorem]{Definition}

\renewcommand{\S}{\text{Split}}

\newcommand{\R}{\text{Rel}}

\newcommand{\ghost}[1]{}

\title{Real Reliability Roots of Simple Graphs are Dense in $(-1,0)$}
\author{Mohamed Omar}
\address{Department of Mathematics \& Statistics. York University. 4700 Keele St. Toronto, Canada. M3J 1P3}
\email{omarmo@yorku.ca}

\date{\today}


\begin{document}

\begin{abstract}
We prove that the closure of the real roots of all-terminal reliability polynomials of simple graphs is exactly $[-1,0] \cup \{1\}$, resolving a conjecture of Brown and McMullin and refining the corresponding density result for multigraphs due to Brown and Colbourn. The crux of the proof is demonstrating that real reliability roots of edge-substitution graphs $G[H]$, where $G$ ranges over connected multigraphs and $H$ ranges over complete graphs missing an edge, are dense.
\end{abstract}

\maketitle

\section{Introduction}
For a connected graph $G$ with vertex set $V(G)$ and edge set $E(G)$, its all-terminal reliability polynomial is
\[
\R(G;q) := \sum_{A \subseteq E(G)} (1-q)^{|A|}q^{|E(G)|-|A|},
\]
where the sum is taken over all subsets $A$ of the edge set for which the subgraph of $G$ with vertex set $V(G)$ and edge set $A$ is connected. This polynomial is a classical invariant in network reliability theory and a standard model of robustness under independent edge failures; see, for example, \cite{BrownSurvey2021,Colbourn1987,PerezRoses2018}. Probabilistically, $\R(G;q)$ is the probability that $G$ remains connected when each edge fails independently with probability $q \in [0,1]$.

From a combinatorial perspective, it is natural to study the real roots of the reliability polynomial for the same reason one studies the zeros of other combinatorial polynomials such as chromatic polynomials, independence polynomials, matching polynomials, and related graph enumerators \cite{ChudnovskySeymour2007,DongKohTeo2005,Godsil1993}. In many settings, real-rootedness is not merely an analytic curiosity; it is a mechanism that forces regularity in the underlying sequence of combinatorial data. For instance, it is classically known that if a polynomial $f(x)=\sum_{k} a_k x^k$ has only real zeros and $a_k\ge 0$ for all $k$, then  log-concavity and often unimodality of the coefficient sequence follow as well \cite{Comtet1974}. This is exactly the sort of phenomenon that has made root geometry so useful elsewhere in algebraic and enumerative combinatorics. Thus, even without knowing in advance that reliability polynomials should enjoy the same behavior as chromatic or independence polynomials, it is mathematically important to ask whether their real-rootedness, or the real-rootedness of a natural transform or specialization, would force analogous inequalities for the combinatorial quantities they encode.

More broadly, this question now sits inside a much richer circle of ideas than the classical theory of graph polynomials alone. Real-rootedness is closely tied to stability and the half-plane property \cite{ChoeOxleySokalWagner2004} and to the emerging theory of Lorentzian polynomials \cite{BrandenHuh2020}; in matroid theory and related settings, these ideas have led to deep log-concavity theorems and Hodge-theoretic explanations for coefficient inequalities \cite{AdiprasitoHuhKatz2018,Huh2015}. From that viewpoint, the reliability polynomial is interesting not just as a single variable invariant, but as a possible shadow of a more rigid multivariate structure. If such a structure exists, then one might hope for the same kind of consequences that have recently appeared elsewhere: log-concavity, Mason-type inequalities, and comparison principles that are difficult to see at the purely enumerative level. So understanding real roots of reliability polynomials allows us to determine whether connectivity-based graph enumeration belongs to this central framework linking zeros, convexity-type inequalities, and deeper combinatorial structure.

The global picture for roots of the polynomials $\R(G;q)$ over all graphs $G$ is rather subtle. Some families are completely real-rooted: for example, trees and cycles have only real reliability roots \cite{Colbourn1987}. Brown and Colbourn proved that every connected multigraph has a subdivision whose reliability polynomial is real-rooted \cite{BrownColbourn1992}. At the same time, the Brown-Colbourn Conjecture asserted that the complex roots of $\R(G;q)$ are confined to a specific disk. This was disproven by Royle and Sokal \cite{RoyleSokal2004}. Brown and Mol later proved the first nontrivial general upper bound on the modulus of reliability roots while also exhibiting simple graphs with roots of modulus greater than $1$ \cite{BrownMol2017}. All in all, the geometry of the complex roots is quite complicated, even though certain graph families are exceptionally well-behaved.

For real roots, the cleanest result is in the multigraph setting. Brown and Colbourn showed that the closure of the set of real reliability roots of connected multigraphs is exactly $[-1,0]\cup\{1\}$ \cite{BrownColbourn1992}. What is not known is whether the same closure statement holds for \emph{simple} graphs. A recent result of Brown and McMullin makes substantial progress in this direction: they prove (among other things) that the real roots of reliability polynomials of simple graphs are dense in the interval $[\beta,0]$, where $\beta \approx -0.5707202942$ \cite{BrownMcMullin2026}. They further conjecture that the full closure in the simple graph case should still be $[-1,0]\cup\{1\}$. This is especially delicate near $-1$: Brown and DeGagn\'e proved that $-1$ is not a reliability root of any simple graph, so in the simple graph setting it can only occur as a limit point \cite{BrownDeGagne2020}. The main theorem of this article is resolving this conjecture.

\begin{theorem}\label{thm:main}
Let $\mathcal{R}$ be the set of real roots of all-terminal reliability polynomials of simple graphs. Then $\overline{\mathcal{R}}=[-1,0] \cup \{1\}$.
\end{theorem}

Our article is organized as follows. We start with preliminaries, reintroducing several auxiliary graph polynomials and rational functions from the literature that are important for establishing Theorem~\ref{thm:main}. Of particular note is the \emph{virtual edge interaction} of a graph; its behavior is the local driver of the proof of the main theorem. We prove our main results in Section~\ref{sec:main}. We end with future directions in Section~\ref{sec:future}.

\section{Preliminaries}\label{sec:prelim}
Following \cite{BrownMol2017}, for a connected graph $H$ with two special vertices $u,v$ called \emph{terminals}, we define the \emph{split reliability polynomial} of $H$ as
\[
\S(H;q)=\sum_{A \subseteq E(H)} (1-q)^{|A|} q^{|E(H)|-|A|},
\]
where the sum here runs over all $A \subseteq E(H)$ such that the subgraph of $H$ with vertex set $V(H)$ and edge set $A$ has exactly two connected components, one containing $u$ and one containing $v$. The corresponding \emph{virtual edge interaction} of $H$ is
\[
\hat{y}_{H}(q) := \frac{\R(H;q)}{\S(H;q)}+1,
\]
which is defined whenever $\S(H;q)$ is not identically $0$. The class of graphs we will be studying these polynomials on are built from the family $\{H_n\}_{n \geq 3}$ of complete graphs with an edge removed.
\begin{definition}
For $n \geq 3$, let $e=uv$ be an edge in the complete graph $K_n$. Define $H_n:=K_n \backslash e$, with terminals $u$ and $v$. 
\end{definition}
For convenience we let $H_2$ be the graph with two isolated vertices $u,v$. Thus $\R(H_2;q)=0$ and $\S(H_2;q)=1$. We abbreviate the following polynomials evaluated on the graphs $K_n$ and $H_n$, as they will be written frequently:
\[
C_n(q)=\R(K_n;q), \qquad R_n(q)=\R(H_n;q), \qquad S_n(q)=\S(H_n;q), \qquad \hat{y}_n(q)=\hat{y}_{H_n}(q).
\]
We have the following recurrences for these polynomials; see \cite{BrownMol2017}.

\begin{proposition}\label{prop:recurrences}
The following recurrences hold for $n \geq 2$:
\begin{equation}\label{eq:complete-rec}
C_n(q)=1-\sum_{a=1}^{n-1} \binom{n-1}{a-1}C_a(q)q^{a(n-a)},   
\end{equation}
\begin{equation}\label{eq:split-rec}
S_n(q)=\sum_{a=1}^{n-1} \binom{n-2}{a-1} C_a(q)C_{n-a}(q)q^{a(n-a)-1},
\end{equation}
\begin{equation}\label{eq:rminus-rec}
R_n(q)=1 - \sum_{a=1}^{n-1} \binom{n-2}{a-1} C_a(q) q^{a(n-a)-1} - \sum_{a=3}^{n-1} \binom{n-2}{a-2} R_a(q) q^{a(n-a)}.
\end{equation}
Furthermore $C_1(q)=1$.
\end{proposition}

The graphs whose real reliability roots play a hallmark role in establishing Theorem~\ref{thm:main} come from edge-substitutions involving the family $\{H_n\}_{n \geq 3}$, as described in \cite{BrownMol2017}.

\begin{definition}
Let $G$ be a connected graph, and let $H$ be a graph with terminals $u$ and $v$. The \emph{edge-substitution graph} $G[H]$ is obtained by replacing each edge $ab$ of $G$ by a copy of $H$ and identifying $u,v$ with $a,b$ respectively. We refer to $H$ as the \emph{gadget} of $G[H]$.
\end{definition}

\begin{figure}[t]
\centering
\begin{tikzpicture}[scale=1, every node/.style={circle,fill=black,inner sep=1.6pt}]
  \node (Ga) at (0,0.8) {};
  \node (Gb) at (0,-0.8) {};
  \draw[bend left=24,line width=0.6pt] (Ga) to (Gb);
  \draw[bend right=24,line width=0.6pt] (Ga) to (Gb);
  \node[draw=none,fill=none,left=3pt] at (Ga) {$a$};
  \node[draw=none,fill=none,left=3pt] at (Gb) {$b$};
  \node[draw=none,fill=none] at (0,-1.8) {$G$};

  \node (Hu) at (4,0) {};
  \node (Hv) at (6,0) {};
  \node (Hc) at (5,0.95) {};
  \node (Hd) at (5,-0.95) {};
  \draw[line width=0.6pt] (Hu) -- (Hc);
  \draw[line width=0.6pt] (Hu) -- (Hd);
  \draw[line width=0.6pt] (Hv) -- (Hc);
  \draw[line width=0.6pt] (Hv) -- (Hd);
  \draw[line width=0.6pt] (Hc) -- (Hd);
  \node[draw=none,fill=none,left=3pt] at (Hu) {$u$};
  \node[draw=none,fill=none,right=3pt] at (Hv) {$v$};
  \node[draw=none,fill=none,above=2pt] at (Hc) {$c$};
  \node[draw=none,fill=none,below=2pt] at (Hd) {$d$};
  \node[draw=none,fill=none] at (5,-1.8) {$H_4=K_4 \backslash e$};

  \node (Sa) at (9,0) {};
  \node (Sb) at (13,0) {};

  \node (U1) at (10.2,1.15) {};
  \node (U2) at (11.8,1.15) {};
  \draw[line width=0.6pt] (Sa) -- (U1);
  \draw[line width=0.6pt] (Sa) -- (U2);
  \draw[line width=0.6pt] (Sb) -- (U1);
  \draw[line width=0.6pt] (Sb) -- (U2);
  \draw[line width=0.6pt] (U1) -- (U2);

  \node (L1) at (10.2,-1.15) {};
  \node (L2) at (11.8,-1.15) {};
  \draw[line width=0.6pt] (Sa) -- (L1);
  \draw[line width=0.6pt] (Sa) -- (L2);
  \draw[line width=0.6pt] (Sb) -- (L1);
  \draw[line width=0.6pt] (Sb) -- (L2);
  \draw[line width=0.6pt] (L1) -- (L2);

  \node[draw=none,fill=none,left=3pt] at (Sa) {$a$};
  \node[draw=none,fill=none,right=3pt] at (Sb) {$b$};
  \node[draw=none,fill=none,above=2pt] at (U1) {$c_1$};
  \node[draw=none,fill=none,above=2pt] at (U2) {$d_1$};
  \node[draw=none,fill=none,below=2pt] at (L1) {$c_2$};
  \node[draw=none,fill=none,below=2pt] at (L2) {$d_2$};
  \node[draw=none,fill=none] at (11,-2.0) {$G[H_4]$};
\end{tikzpicture}
\caption{An example of the substitution construction with $n=4$. On the left, the multigraph $G$ has vertices $a$ and $b$ joined by two parallel edges. In the middle, the gadget $H_4=K_4 \backslash e$ has terminals $u$ and $v$ (so $e=uv$). On the right, each edge of $G$ is replaced by a fresh copy of $H_4$, and in each copy the terminal $u$ is identified with $a$ while the terminal $v$ is identified with $b$.}\label{fig:substitution}
\end{figure}
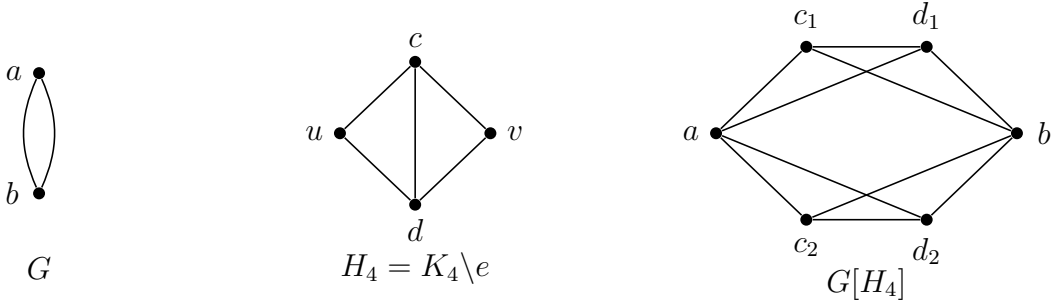

If $H$ is a simple graph and its terminals are not adjacent, then $G[H]$ is a simple graph even when $G$ is a multigraph. An example of an edge-substitution graph is given in Figure~\ref{fig:substitution}.

The following relations regarding reliability polynomials of edge-substitution graphs, due to Brown and Mol \cite{BrownMol2017}, will be pertinent for our use.

\begin{proposition}\label{prop:substitutionpolynomials}
Let $G$ be a connected graph on $m$ edges, and $F(G,z)=\sum_i F_i z^i$ where $F_i$ counts the number of $i$-edge subsets of $G$ whose deletion leaves a connected graph. Then:

\begin{enumerate}
\item for every connected two-terminal graph $H$,
\[
\R(G[H];q)=(\R(H;q))^m \cdot F\left(G,\frac{\S(H;q)}{\R(H;q)}\right),
\]
\item if $r \neq 1$ is a reliability root of $G$, then every solution of
\[
\S(H;q)=\frac{r}{1-r}\R(H;q)
\]
is either a reliability root of $H$ or a reliability root of $G[H]$,
\item when $\S(H;q) \neq 0$, the equation in condition (2) is equivalent to
\[
\hat{y}_H(q)=1/r.
\]
\end{enumerate}
\end{proposition}

We end our preliminaries with a technical lemma that will be used extensively to understand the limiting behavior of the combinatorial functions $C_n(q),R_n(q),S_n(q),\hat{y}_n(q)$.

\begin{lemma}\label{lem:kernel}
Let $K \subset (-1,0)$ be compact. Fix an integer $s$ and let $b_{N,a}(q)$ be real-valued functions indexed by natural numbers $a,N$ with $N \geq 2$ and $1 \leq a \leq N-1$. Suppose
\[
\sup \ \{ |b_{N,a}(q)| \colon N \geq 2, \ 1 \leq a \leq N-1, \ q \in K \} <\infty.
\]
Then
\[
\sup_{q \in K} \left| \sum_{a=1}^{N-1} \binom{N}{a}b_{N,a}(q) \cdot q^{a(N-a)+s} \right| \to 0 \qquad \text{ as } N \to \infty.
\]
\end{lemma}
\begin{proof}
Set 
\[
B = \sup \ \{ |b_{N,a}(q)| \colon N \geq 2, \ 1 \leq a \leq N-1, \ q \in K \}.
\]
Let $c_s=\sup\{|q|^s \colon q \in K\}$ and let $\rho=\sup\{|q| \colon q \in K\}$. Since $K$ is compact and $0 \notin K$, $c_s$ and $\rho$ are finite. Furthermore, $0<\rho<1$ since $K \subset (-1,0)$.

Choose an integer $A \geq 2$ such that $2\rho^A<1$. Consider $N \geq 2A$. Split the sum in consideration based on the index $a$: in the central range we will have $A \leq a \leq N-A$, then in two boundary ranges we will have $1 \leq a \leq A-1$ and $N-(A-1) \leq a \leq N-1$. For any $a$ in the central range, we have $a(N-a) \geq A(N-A)$ so since $0<\rho<1$ this implies
\[
|q|^{a(N-a)+s} \leq c_s \rho^{a(N-a)} \leq c_s \rho^{A(N-A)}.
\]
Consequently for $q \in K$ we have,
\begin{align*}
\left| \sum_{a=A}^{N-A} \binom{N}{a} b_{N,a}(q) \cdot q^{a(N-a)+s} \right| &\leq \sum_{a=A}^{N-A} \binom{N}{a} |b_{N,a}(q)||q|^{a(N-a)+s} \\ &\leq Bc_s \rho^{A(N-A)} \sum_{a=A}^{N-A} \binom{N}{a} \\
&\leq Bc_s2^N\rho^{A(N-A)}.
\end{align*}
On the lower of the boundary ranges, fix $1 \leq a \leq A-1$. Then for $q \in K$,
\[
\left|\binom{N}{a}b_{N,a}(q) \cdot q^{a(N-a)+s}\right| \leq Bc_s\binom{N}{a}\rho^{a(N-a)} \leq Bc_sN^a\rho^{aN-a^2}.
\]
Finally on the upper of the boundary ranges, the same upper bound applies as in the lower of the boundary ranges, because when $a$ is replaced by $N-a$, both $\binom{N}{a}$ and $a(N-a)$ stay the same. Therefore for any $q \in K$ we have 
\begin{align*}
\left| \sum_{a=1}^{N-1} \binom{N}{a}b_{N,a}(q) \cdot q^{a(N-a)+s} \right|  &\leq 2Bc_s \cdot \left(\sum_{a=1}^{A-1} N^a\rho^{aN-a^2}\right) + Bc_s2^N\rho^{A(N-A)} \\
&= 2Bc_s \cdot \left(\sum_{a=1}^{A-1} N^a\rho^{aN-a^2}\right) + Bc_s\rho^{-A^2}(2\rho^A)^N.
\end{align*}
The right hand side is independent of $q \in K$. The first sum is bounded above by $c(A,\rho) \cdot N^{A-1}\rho^N$ where $c(A,\rho)$ is a constant dependent only on $A$ and $\rho$, because the exponent of $\rho$ satisfies the inequality $aN-a^2 \geq N-(A-1)^2$, on the range $1 \leq a \leq A-1$. Hence the first sum vanishes as $N \to \infty$. The final term also vanishes as $N \to \infty$ by the choice of $A$ because $0<2\rho^A<1$. The result follows.
\end{proof}

\section{Main Results}\label{sec:main}

An important technical step toward Theorem~\ref{thm:main} is to show that, for a given compact $K \subset (-1,0)$, one can choose $n$ large enough so that $\hat{y}_n(q)$ is sufficiently negative on $K$. This is the content of the next proposition.

\begin{proposition}\label{prop:yn}
Let $K \subset (-1,0)$ be compact. Then there exists an integer $N$ such that $\hat{y}_N(q)<-1$ for every $q \in K$.
\end{proposition}

To see how Proposition~\ref{prop:yn} leads to the proof of Theorem~\ref{thm:main}, we consider an example. To prove that real reliability roots are dense in $(-1,0)$ we need to prove that any open interval $I \subset (-1,0)$ contains a real reliability root. Suppose we were given the open interval $I=(-0.5,-0.3)$. Let $K=[-0.40,-0.35] \subset I$. Since $K \subset (-1,0)$ is compact, Proposition~\ref{prop:yn} certifies that there is an integer $N$ such that $\hat{y}_N(q)<-1$ for all $q \in K$. It happens to be the case that $N=5$ works for this specific $K$. Indeed, a direct computation using the formulas in Proposition~\ref{prop:recurrences} gives
\[
\R(H_5;q)=(1-q)^4(18q^5+24q^4+18q^3+10q^2+4q+1),
\]
\[
\S(H_5;q)=2q^3(1-q)^3(12q^3+9q^2+3q+1),
\]
so
\[
\hat{y}_5(q)=
\frac{(1+q)(6q^5+6q^4+6q^3+4q^2+2q+1)}
{2q^3(12q^3+9q^2+3q+1)}.
\]
Numerically, one can then certify that 
\[
\hat{y}_5([-0.40,-0.35])=[a,b]\subset(-\infty,-1).
\]
where $a \approx -9.6454$ and $b \approx -6.8930$. Now real reliability roots of multigraphs are dense in $(-1,0)$ so their reciprocals are dense in $(-\infty,-1)$. This means we can find a real reliability root $r$ of a multigraph $G$ so that $1/r \in (a,b)$. For completeness, we will find an explicit such $r$ and $G$. If one takes $G=\mathcal{C}_8$, the cycle graph on $8$ vertices, then connected subgraphs indexing the sum $\R(\mathcal{C}_8,q)$ are $\mathcal{C}_8$ itself and $8$ paths, each obtained by deleting some edge from $\mathcal{C}_8$. From this, 
\[
\R(\mathcal{C}_8;q)=(1-q)^8+8q(1-q)^{7}=(1-q)^{7}(1+7q),
\]
So $G$ has real reliability root $r=-1/7$. Because $1/r=-7$ lies in the image of $\hat{y}_5$ on $K$, there exists $q_0 \in K$ such that $\hat{y}_5(q_0)=1/r$ (an approximate numerical value for $q_0$ in this case is $q_0\approx -0.39713$). By conditions (2) and (3) in Proposition~\ref{prop:substitutionpolynomials}, $q_0$ is a real reliability root of the simple graph $H_5$ or the simple edge-substitution graph $\mathcal{C}_8[H_5]$. Since $q_0 \in K \subset I$, we have found a real reliability root in our open interval $I$.

This is a small-scale picture of the main theorem. To show roots can come arbitrarily close to any point in $(-1,0)$, select a nonempty open interval $I \subset (-1,0)$ and find a real reliability root in $q_0 \in I$ as follows. Select a compact $K \subset I$. Use Proposition~\ref{prop:yn} to find an $N$ so that $\hat{y}_N$ sends $K$ into $(-\infty,-1)$. Then, select a reciprocal target $1/r \in (-\infty,-1)$ in the image of $\hat{y}_N$ where $r$ is the real reliability root of a multigraph $G$. Finally, pull this back to get a real reliability root $q_0 \in K \subset I$ such that $\hat{y}_N(q_0)=1/r$. By (2) and (3) of Proposition~\ref{prop:substitutionpolynomials}, $q_0$ is a real reliability root of a simple graph.

\begin{proof}[Proof of Proposition~\ref{prop:yn}]
The proof has three parts. First, we use Proposition~\ref{prop:recurrences} and Lemma~\ref{lem:kernel} to prove the families of functions $\{C_n\}_{n \geq 1}$ and $\{R_n\}_{n \geq 1}$ are uniformly bounded on $K$. Then, we use this to show that uniformly on $K$, 
\[
C_n(q) \to 1, \qquad R_n(q) \to 1, \qquad q^{2-n}S_n(q) \to 2.
\]
It will then follow that $2q^{n-2}\hat{y}_n(q) \to 1$ uniformly on $K$ and this will lead to a sufficiently large $N$ for which $\hat{y}_N(q)<-1$ for all $q \in K$.

We start by showing $\{C_n\}_{n \geq 1},\{R_n\}_{n \geq 1}$ are uniformly bounded on $K$. For $n \geq 2$, let
\[
X_n:=\sup_{q \in K} |C_n(q)|, \qquad Y_n:=\sup_{q \in K} |R_n(q)|, \qquad M_n:=\max\{X_k,Y_k \colon 1 \leq k \leq n\}.
\]
From \eqref{eq:complete-rec} of Proposition~\ref{prop:recurrences}, and the fact that $\binom{n-1}{a-1}=\binom{n}{a}\frac{a}{n}$, we have that for any $q \in K$,
\[
|C_n(q)| \leq 1 + M_{n-1} \sum_{a=1}^{n-1} \binom{n}{a} \frac{a}{n} |q|^{a(n-a)}=1+M_{n-1}\sum_{a=1}^{n-1} \binom{n}{a} \frac{a}{n} (-1)^{a(n-a)} q^{a(n-a)},
\]
the latter equality because $q<0$. Taking the supremum over $q \in K$ we have $X_n \leq 1+M_{n-1}\alpha_n$ where 
\[
\alpha_n=\sup_{q \in K} \sum_{a=1}^{n-1} \binom{n}{a} \frac{a}{n} (-1)^{a(n-a)} q^{a(n-a)}.
\] 
In $\alpha_n$ we can replace the sum by its absolute value because the sum is positive since $q<0$. Therefore, we can apply  Lemma~\ref{lem:kernel} with $b_{n,a}(q)=\frac{a}{n}(-1)^{a(n-a)}$, $N=n$ and $s=0$. We see that $|b_{n,a}(q)|$ is uniformly bounded by $1$ for all $n,a$. Therefore $\alpha_n \to 0$ as $n \to \infty$.

Applying a similar process, with the observation that 
\[
\binom{n-2}{a-1}=\frac{a(n-a)}{n(n-1)}\binom{n}{a} \qquad{ \text{and} } \qquad \binom{n-2}{a-2}=\frac{a(a-1)}{n(n-1)}\binom{n}{a}
\]
we have $Y_n \leq 1+M_{n-1}\beta_n$ where
\[
\beta_n := \sup_{q \in K} \sum_{a=1}^{n-1} \binom{n}{a}\frac{a(n-a)}{n(n-1)}(-1)^{a(n-a)-1} q^{a(n-a)-1} + \sup_{q \in K} \sum_{a=3}^{n-1} \binom{n}{a}\frac{a(a-1)}{n(n-1)}(-1)^{a(n-a)} q^{a(n-a)}. 
\]
Apply Lemma~\ref{lem:kernel} to the first of the two sums with $N=n,s=-1,b_{n,a}(q)=\frac{a(n-a)}{n(n-1)}(-1)^{a(n-a)}$, and the same lemma to the second sum with $N=n,s=0,b_{n,a}(q)=\frac{a(a-1)}{n(n-1)}(-1)^{a(n-a)}$ for $3 \leq a \leq n-1$ and $b_{n,a}(q)=0$ for $1 \leq a \leq 2$, this yields $\beta_n \to 0$ as $n \to \infty$.

For any $n$ let $\gamma_n=\max\{\alpha_n,\beta_n\}.$ Then we have $\gamma_n\to 0$. Furthermore, for $n\ge 2$, we have $X_n\le 1+M_{n-1}\gamma_n$ and $Y_n\le 1+M_{n-1}\gamma_n$, so $M_n=\max\{M_{n-1},X_n,Y_n\}\le \max\{M_{n-1},\,1+M_{n-1}\gamma_n\}.$ Choose $N \geq 2$ such that \(\gamma_n\le \tfrac12\) for all \(n\ge N\). Then for all \(n\ge N\),
\[
M_n\le \max\left\{M_{n-1},\,1+\frac12 M_{n-1}\right\}.
\]
Set
\[
C=\max\{M_1,\dots,M_{N-1},2\}.
\]
We prove by induction that $M_n\le C$ for all $n$. This is immediate for $n\le N-1$. If $n\ge N$ and $M_{n-1}\le C$, then
\[
X_n\le 1+\frac12 M_{n-1}\le 1+\frac12 C\le C,
\]
the last inequality because $C \geq 2$. Similarly $Y_n\le C$. Therefore $M_n=\max\{M_{n-1},X_n,Y_n\}\le C.$ Thus $M_n\le C$ for all $n$. Since $0\le X_n,Y_n\le M_n$, it follows that the polynomials $\{C_n\}_{n \geq 1},\{R_n\}_{n \geq 1}$ are uniformly bounded on $K$.

Next we show that
\[
C_n(q) \to 1, \qquad R_n(q) \to 1, \qquad q^{2-n}S_n(q) \to 2,
\]
uniformly on $K$ as $n \to \infty$. From Proposition~\ref{prop:recurrences},
\[
C_n(q)-1=-\sum_{a=1}^{n-1}\binom{n}{a}\frac{a}{n}C_a(q)q^{a(n-a)}.
\]
The coefficients $-\frac{a}{n}C_a(q)$ are bounded in absolute value by $C$, so Lemma~\ref{lem:kernel} implies $|C_q(n)-1| \to 0$ uniformly on $K$, and hence $C_q(n) \to 1$ uniformly on $K$. Similarly,
\[
R_n(q)-1 = -\sum_{a=1}^{n-1} \binom{n}{a} \frac{a(n-a)}{n(n-1)}C_a(q)q^{a(n-a)-1} - \sum_{a=3}^{n-1} \binom{n}{a} \frac{a(a-1)}{n(n-1)}R_a(q)q^{a(n-a)}.
\]
The coefficients $-\frac{a(n-a)}{n(n-1)}C_a(q)$ and $-\frac{a(a-1)}{n(n-1)}R_a(q)$ are both bounded above by $C$ in absolute value, so applying Lemma~\ref{lem:kernel} with $s=-1$ to the first sum and $s=0$ to the second we get $R_n(q) \to 1$ uniformly on $K$. Finally, to show $q^{2-n}S_n(q) \to 2$, start by isolating the $a=1$ and $a=n-1$ terms in the sum \eqref{eq:split-rec} in Proposition~\ref{prop:recurrences} and multiply by $q^{2-n}$. Together with the fact $C_1(q)=1$, this yields,
\begin{align*}
q^{2-n}S_n(q) &= 2C_{n-1}(q) + \sum_{a=2}^{n-2} \binom{n-2}{a-1} C_a(q)C_{n-a}(q)q^{a(n-a)-1+2-n} \\
&= 2C_{n-1}(q) + \sum_{a=2}^{n-2} \binom{n-2}{a-1} C_a(q)C_{n-a}(q)q^{(a-1)(n-a-1)}.
\end{align*}
Reindexing the sum with $N=n-2$ and $b=a-1$, we have 
\[
\left|q^{2-n}S_n(q)-2C_{n-1}(q)\right|=\left|\sum_{b=1}^{N-1} \binom{N}{b} C_{b+1}(q)C_{N+1-b}(q)q^{b(N-b)}\right|.
\]
By Lemma~\ref{lem:kernel}, this vanishes because $|C_{b+1}(q)C_{N+1-b}(q)|$ is uniformly bounded by $C^2$ on $K$. Since $C_{n-1}(q)$ converges to $1$ uniformly on $K$, it follows that $q^{2-n}S_n(q) \to 2$ uniformly on $K$. 

We now find a positive integer $N$ such that $\hat{y}_N(q)<-1$ for all $q \in K$. Since $q^{2-n}S_n(q) \to 2$ uniformly on $K$, we can select $N_0$ so that for all $q \in K$
\[
|q^{2-n}S_n(q)-2|<1 \qquad \forall n \geq N_0.
\]
For such $n$, we see that $q^{2-n}S_n(q)$, and hence $S_n(q)$, are nonzero. So $\hat{y}_n(q)$ is defined. We then have
\begin{align*}
|2q^{n-2}\hat{y}_n(q)-1| &=\left|2q^{n-2} \cdot \left(\frac{R_n(q)}{S_n(q)}+1\right)-1\right|\\
&= \left| \frac{2R_n(q)}{q^{2-n}S_n(q)} - 1 + 2q^{n-2} \right| \\
&\leq \frac{|2R_n(q)-q^{2-n}S_n(q)|}{|q^{2-n}S_n(q)|} + 2|q|^{n-2} \\
&\leq |2R_n(q)-q^{2-n}S_n(q)|+2|q|^{n-2} \\
&\leq |2R_n(q)-2|+|2-q^{2-n}S_n(q)|+2|q|^{n-2},
\end{align*}
the second-to-last inequality holding because $|q^{2-n}S_n(q)-2|<1$ for $n \geq N_0$ and $q \in K$ implies $q^{2-n}S_n(q)>1$ for the same values of $n$ and $q$. Now $2R_n(q) \to 2$ and $q^{n-2}S_n(q) \to 2$ uniformly on $K$. Moreover since $K$ is a compact subset of $(-1,0)$, there is a universal $\rho \in (0,1)$ so that $|q|\leq \rho$ for $q \in K$ so $2|q|^{n-2} \to 0$ uniformly on $K$. Consequently
\[
2q^{n-2}\hat{y}_n(q) \to 1
\]
uniformly on $K$. So there exists $N_1$ such that for all $q \in K$,
\[
|2q^{n-2}\hat{y}_n(q)-1|<\frac{1}{2} \qquad \forall n \geq N_1.
\]
Moreover, there exists $N_2$ such that for $n \geq N_2$, 
\[
\rho^{n-2}<\frac{1}{4} \qquad \forall n \geq N_2,
\]
because $0<\rho<1$. 

Select an odd integer $N \geq 3$ so that $N \geq \max(N_0,N_1,N_2)$. We claim that if $q \in K$ then $\hat{y}_N(q)<-1$. Indeed select $q \in K$. Since $N \geq N_0,N_1$, we have $2q^{N-2}\hat{y}_N(q)>1/2$. Now since $q<0$ and $N$ is odd, $q^{N-2}<0$. Therefore
\[
\hat{y}_N(q)<\frac{1}{4q^{N-2}} \leq \frac{-1}{4\rho^{N-2}}.
\]
Finally, $N \geq N_2$ implies $\hat{y}_N(q)<-1$. 
\end{proof}

We now prove our main theorem.

\begin{proof}[Proof of Theorem~\ref{thm:main}]
Let $\mathcal{R}$ be the set of real reliability roots. By Brown and Colbourn \cite{BrownColbourn1992}, the closure of the real reliability roots of multigraphs is $[-1,0]\cup\{1\}$. Since simple graphs are multigraphs, we get the inclusion $\overline{\mathcal{R}} \subseteq [-1,0] \cup \{1\}$. Now suppose that  every nonempty open interval $I \subseteq (-1,0)$ contains the real reliability root of a simple graph. Combined with the fact that every connected simple graph with at least one edge has a real reliability root $1$, this implies $[-1,0] \cup \{1\} \subseteq \overline{\mathcal{R}}$. Therefore $\overline{\mathcal{R}}=[-1,0] \cup \{1\}$.

It therefore suffices to prove that if $I \subset (-1,0)$ is a nonempty open interval then $I$ contains the real reliability root of a simple graph. Since $I$ is nonempty, there is a compact interval $K \subset I$ with nonempty interior. By Proposition~\ref{prop:yn}, there is an integer $N$ such that $\hat{y}_N(q)<-1$ for all $q \in K$. In particular, $S_N(q) \neq 0$ for every $q \in K$, so $\hat{y}_N$ is a continuous real-valued function on $K$. Let $J$ be any nonempty open interval with $J \subset K$. Then $\hat{y}_N$ is nonconstant on $J$ because it is a rational function that is nonconstant on its domain (indeed, at $q=0$ we have $R_N(0)=1$ and $S_N(0)=0$ so as $q \to 0$ from the left, $\hat{y}_N(q)$ is nonconstant). Now $\hat{y}_N$ being nonconstant on the nonempty open interval $J$ and $\hat{y}_N$ being continuous in $J$ (since it is continuous on $K \supset J$) implies the image of $J$ under $\hat{y}_N$ is a nondegenerate interval. Together with the fact that $\hat{y}_N(q)<-1$ for all $q \in J$, the image of $J$ under $\hat{y}_N$ is a nondegenerate subinterval of $(-\infty,-1)$. By Brown and Colbourn \cite{BrownColbourn1992}, the real reliability roots of multigraphs are dense in $(-1,0)$, so there exists a real reliability root $r \in (-1,0)$ of a connected multigraph $G$ such that $1/r$ is in the image of $J$ under $\hat{y}_N$. Then, select a $q_0 \in J$ with $\hat{y}_N(q_0)=1/r$. Since $\hat{y}_N(q_0)<-1$, $S_n(q_0) \neq 0$ so by (2) and (3) of Proposition~\ref{prop:substitutionpolynomials}, $q_0$ is either a real reliability root of $H_N$ or $G[H_N]$. Both of these graphs are simple, so $q_0 \in J \subset K \subset I$ is a real reliability root of a simple graph.
\end{proof}
We remark that in the proof of Theorem~\ref{thm:main}, $q_0$ has to be a real reliability root of $G[H_N]$. Indeed, if it was a real reliability root of $H_N$ instead, then $\hat{y}_N(q_0)=R_N(q_0)/S_N(q_0)+1=1$, contradicting $\hat{y}_N(q_0)<-1$. So in fact, not only are real reliability roots dense in $[-1,0] \cup \{1\}$, taking only the roots of edge-substitution graphs of the form $G[H]$, where $G$ is a connected multigraph and $H$ is a complete graph with an edge removed, creates a desired dense set.

\section{Future Directions}\label{sec:future}

Theorem~\ref{thm:main} settles the closure problem for real reliability roots of simple graphs, but it leaves several natural quantitative questions. Our proof is qualitative: given an open interval $I \subset (-1,0)$, it produces an integer $N$ and a connected multigraph $G$ such that the associated edge-substitution graph $G[H_N]$ has a real reliability root in $I$. However, it does not identify the smallest such $N$, nor does it control how large the multigraph $G$ must be in order to force a root into a prescribed subinterval. Determining effective bounds of this sort would make the density theorem considerably more concrete. At a more structural level, it would be valuable to understand which other simple two-terminal gadgets $H$ have the property that the associated virtual edge interaction $\hat{y}_H$ maps a prescribed interval in $(-1,0)$ into $(-\infty,-1)$, since our argument shows that this is exactly the mechanism that transfers dense sets of multigraph roots to simple graph roots. 

A second direction concerns the fine structure near $q=-1$ and the broader distribution of real roots. Brown and DeGagn\'e showed that $-1$ is never itself a reliability root of a simple graph \cite{BrownDeGagne2020}, so it is natural to seek explicit simple graph families whose real roots converge to $-1$ and to understand the rate of that convergence. More generally, combined with the result of Brown and McMullin that almost every graph has a nonreal reliability root \cite{BrownMcMullin2026}, our theorem suggests asking how many real reliability roots a typical simple graph can have, how often one sees only the trivial root $1$, and which graph families remain genuinely close to real-rooted. Finally, the density phenomenon established here should be compared with coefficient inequality questions for reliability polynomials: even though global real-rootedness fails dramatically, it remains plausible that on suitable graph classes there may be a deeper multivariate, stable, or Lorentzian framework governing these polynomials \cite{BrandenHuh2020,ChoeOxleySokalWagner2004}. Establishing or ruling out such structure would sharpen our understanding of what features of reliability are genuinely enumerative and what features reflect a more rigid geometric theory.

\section*{Acknowledgments}
The author is partially funded by research funds from York University, and NSERC Discovery Grant \#RGPIN-2025-06304.

\end{document}